\documentclass[11pt]{article}

\usepackage{fullpage}
\usepackage{amsthm}
\usepackage{amssymb}
\usepackage{amsmath}
\usepackage{graphicx}
\usepackage{tikz}
\usepackage{mathrsfs}
\usepackage{float}
\usepackage{makecell}
\usepackage{amscd}
\usepackage{multirow}
\usepackage{cite}
\usepackage{verbatim}
\usepackage{hhline}
\usepackage{comment}

\usepackage{pgfplots}
\pgfplotsset{compat=1.15}
\usetikzlibrary{arrows}

\newcommand{\mnf}{\mathfrak{m}}

\renewcommand{\epsilon}{\varepsilon}

\usepackage{relsize}
\newtheorem{theorem}{Theorem}
\newtheorem{lemma}[theorem]{Lemma}

\newtheorem{proposition}[theorem]{Proposition}

\newtheorem{conjecture}[theorem]{Conjecture}
\theoremstyle{definition}

\newtheorem{definition}[theorem]{Definition}
\newtheorem{question}{Question}

\title{Chromatic numbers of flag 3-spheres}
\author{Andrew Newman\thanks{Carnegie Mellon University}}

\begin{document}
\maketitle
\begin{abstract}
A recent conjecture of Chudnovsky and Nevo asserts that flag triangulations of spheres always have linear-sized independent sets, with a precisely conjectured proportion depending on the dimension. For dimensions one and two, the lower bound of their conjecture basically follow from constant bounds on the chromatic number of flag triangulations of $S^1$ and $S^2$. This raises a natural question that does not appear to have been considered: For each $d$ is there a constant upper bound for the chromatic number of flag triangulations of $S^d$? Here we show that the answer to this question is no, and use results from Ramsey theory to construct flag triangulations of 3-spheres on $n$ vertices with chromatic number at least $\widetilde{\Omega}(n^{1/4})$.
\end{abstract}

\section{Introduction}
A flag complex (or alternatively a clique complex) is a simplicial complex whose minimal non-faces all have size 2. Equivalently a complex is flag if it is the maximal simplical complex on its graph. Questions about flag complexes can thus be regarded as questions about higher dimensional properties of graphs since a flag complex is determined by its graph. Of particular interest are flag manifolds, flag complexes that triangulate (closed) manifolds. 

Motivated by the study of $f$-vectors of flag spheres, Chudnovsky and Nevo recently made the following conjecture:
\begin{conjecture}[Chudnovsky--Nevo \cite{ChudnovskyNevo}]
Let $\alpha(n, d)$ denote the minimum size of the largest independent set among all flag triangulations of $d$-spheres on $n$ vertices, then
\[\alpha(n, d) = \left \lceil \frac{n + d - 2}{2d} \right \rceil.\]
\end{conjecture}
As they point out this is trivial for flag triangulations of $1$-sphere, i.e. for cycle graphs. The $n$-cycle always has maximal independent set of size $\left \lceil \dfrac{n - 1}{2}\right \rceil$. Moreover for $d = 2$, the lower bound follows from the four color theorem and they prove a matching upper bound in their paper. They prove the upper bound for $d = 3$, but still in that case it remains to show a matching lower bound. The main result of their paper regarding this conjecture for general dimensions $d$ is 
\begin{theorem}[Chudnovsky--Nevo \cite{ChudnovskyNevo}]\label{CNMainTheorem}
For $d \geq 3$ and $n \geq 2(d + 1)$,
\[\frac{1}{4} n^{1/(d - 1)}\leq \alpha(n, d) \leq (1 + o_d(1)) \frac{2n}{3d}.\]
\end{theorem}
One key distinction between the $d \leq 2$ case with the $d \geq 3$ case is that a linear lower bound on $\alpha(n, d)$ for $d \leq 2$ follows simply by bounding the chromatic number of flag triangulations of circles and 2-spheres. However, Chudnovsky and Nevo do not comment on what bounds on the chromatic number of flag spheres might look like in general. Moreover, while the four color theorem is a well known result that implies that flag 2-spheres always have chromatic number at most 4, the question of the chromatic number of flag 3-spheres, does not appear to have been considered before. In fact, it does not seem to be the case that the possibility of a constant upper bound on the chromatic number of flag 3-spheres had even been ruled out. Here we let $\chi(n, d)$ denote the largest chromatic number among all flag $d$-spheres on $n$ vertices. We show that $\chi(n, 3)$ is unbounded (and therefore $\chi(n, d))$ is unbounded for larger $d$ as well) and give upper and lower bounds for its asymptotics.
\begin{theorem}\label{maintheorem}
As $n$ tends to infinity 
\[\widetilde{\Omega}(n^{1/4}) \leq \chi(n, 3) \leq 4 n^{1/2}.\]
\end{theorem}
Here and throughout $\widetilde{\Omega}$ and $\widetilde{O}$ are the usual Bachmann--Landau notation, but ignoring polylogarithmic factors. 
\section{Background}
Generalizations of the four color theorem to higher dimensional settings have been considered in the literature. Lutz and M\o{}ller \cite{LutzMoller} studied the weak chromatic numbers of skeleta of triangulated manifolds regarded as hypergraphs. More specifically for a pure $k$-dimensional complex, the weak chromatic number is the minimum number of colors required to color the associated $(k + 1)$-uniform hypergraph so that no hyperedge is monochromatic. For a simplicial complex $X$ then there is  a sequence $\chi_1(X) \geq \chi_2(X) \geq \cdots \geq \chi_d(X)$ where $\chi_k(X)$ is the weak chromatic number of the pure $k$-part of $X$. Lutz and M\o{}ller study whether or not certain classes of manifolds admit triangulations $X$ with arbitrarily large $\chi_k(X)$ for different manifold dimensions and values of $k$. 

Around the same time as the work of \cite{LutzMoller}, Heise, Panagiotou, Pikhurko, and Taraz \cite{HPPT} studied these weak chromatic numbers for $k$-complexes that embed in $d$-dimensional space. Further results in this direction were also recently proved by Lee and Nevo \cite{LeeNevo}. 

The question specifically of chromatic numbers restricted to \emph{flag} triangulations does not seem to have been considered. Moreover, the best quantitative bounds for $\chi_k(X)$ within the context of the work of \cite{LutzMoller, HPPT, LeeNevo} are far apart; see Table 1 of \cite{HPPT} which shows sub-logarithmic lower bounds and polynomial upper bounds. 

Questions about largest possible chromatic numbers for simplicial complexes with certain properties also fit into Ramsey Theory. For comparison, a flag 3-sphere is a particular type of $K_5$-free graph, and we recall the best known bounds on the chromatic number of a $K_5$-free graph on $n$ vertices. The best-known lower bound on the off-diagonal Ramsey number $R(5, t)$ is 
\[\Omega\left( \frac{t^3}{(\log t)^{8/3}} \right)\]
due to Bohman and Keevash \cite{BohmanKeevash}. From this we have that for $n$ large enough there is are $K_5$-free graphs with chromatic number $\widetilde{\Omega}(n^{2/3})$. The best-known upper bound on $R(5, t)$ is 
\[O\left( \frac{t^4}{(\log t)^{3}}\right)\]
due to \cite{AjtaiKomlos}. By repeatedly coloring large independent sets with one color we have that a $K_5$-free graph on $n$ vertices has chromatic number $O(n^{3/4})$. This upper bound on the chromatic number for a $K_5$-free graph can also be proved inductively in a way similar to the proof of Proposition \ref{UpperBoundProp} below.

\section{The upper bound}
The proof of the upper bound in Theorem \ref{maintheorem} follows basically the same proof as the proof of the lower bound in Theorem \ref{CNMainTheorem} in \cite{ChudnovskyNevo}. It is a slight improvement though because it gives an upper bound for the chromatic number rather than a lower bound for the largest independent set. Our upper bound follows from the following proposition that holds in all dimensions and for all flag manifolds:
\begin{proposition}\label{UpperBoundProp}
For $d \geq 3$, the chromatic number of a flag $d$-manifold on $n$ vertices is at most $C_d n^{1 - 1/(d - 1)}$ for some constant $C_d$ depending only on the dimension. 
\end{proposition}
\begin{proof}
We proceed by induction. For $d = 3$ we have that if the maximum degree of a flag $d$-manifold is at most $x\sqrt{n}$ (for $x$ constant to be selected later) then by Brooks' Theorem the chromatic number is at most $x \sqrt{n}$. Otherwise, we select a vertex $v$ of degree larger than $x \sqrt{n}$. Now the link of $v$ is necessarily a flag 2-sphere, so it is four colorable. So we can color the link of $v$ with four (new) colors and then delete those vertices. We continue to repeat this process (observing that vertex links always remain planar) until we end up with all vertices of degree at most $x \sqrt{n}$ at which point we apply Brooks' Theorem. After at most $\frac{n}{x \sqrt{n}}$ repetitions of this process we have used at most $\frac{4 \sqrt{n}}{x}$ colors and have at most $x \sqrt{n}$ vertices remaining that we then color with $x \sqrt{n}$ colors. Thus for any $x$ we can color our flag 3-manifold with at most 
\[\left(\frac{4}{x} + x\right)\sqrt{n}\]
colors. This holds for any $x$ so we take $x = 2$ to minimize the number of colors at $4\sqrt{n}$. 

Now inductively if we have a flag $d$-manifold and the maximum degree is at most $xn^{1 - 1/(d - 1)}$ then we apply Brooks' Theorem. Otherwise we can apply the $(d - 1)$-case within a vertex link as before and show that we can color with at most
\[\left(x + \frac{C_{d-1}}{x^{1/(d - 2)}}\right)n^{1 - 1/(d - 1)}\]
colors.

We had that $C_3 = 4$ and by basic calculus for $d \geq 4$ we can take
\[C_d = \left(\frac{C_{d - 1}}{d - 2} \right)^{(d - 2)/(d - 1)} + C_{d-1}\left(\frac{d - 2}{C_{d - 1}}  \right)^{1/(d - 1)}.\]
\end{proof}
We observe that the proof of Proposition \ref{UpperBoundProp} only relies on the four color theorem and the fact that links of faces of dimension $d - 3$ are planar graphs, and the four color theorem is only used to slightly improve $C_d$. We now show by a short random argument that these assumptions alone are not sufficient to substantially improve the upper bound, so if we wish to improve on the the upper bound for chromatic number of flag $d$-manifolds, or even flag $d$-spheres, we'd need to use more of their topological or geometric properties. 
\begin{proposition}\label{UpperBoundProposition}
For any $\delta < 1 - 1/(d - 1)$ and $n$ large enough there is a $d$-dimensional flag complex so that the link of every $(d - 3)$-dimensional face embeds in the 2-sphere and the chromatic number is at least $n^{\delta}$.
\end{proposition}
\begin{proof}
Take $p = n^{-\alpha}$ for $1/(d - 2) > \alpha > 1/(d - 1)$. Then the random clique complex $X \sim X(n, n^{-\alpha})$ (the clique complex of the Erd\H{o}s--R\'{e}nyi random graph, introduced in \cite{KahleRandomClique}) has that the link of every $(d - 3)$-dimensional face is essentially a random clique complex sampled from $X(m, m^{-\alpha/(1 - (d - 2)\alpha)})$ with $m = n^{1-(d - 2)\alpha} \rightarrow \infty$. Since the connectivity threshold for the underlying graph is at $\frac{\log m}{m}$ by the classic Erd\H{o}s--R\'enyi result and 
\[\frac{\alpha}{1 - (d - 2)\alpha} > 1,\]
almost all vertex links are forests. We delete the $o(n)$ of the vertices that have links that are not forests (this also brings the dimension of the complex down to $d$). By a first moment argument $X$ has no independent set of size larger than 
$C n^{\alpha}\log n$, for some constant $C$, from which the claim follows.
\end{proof}
\section{The lower bound}
The construction for the lower bound is based on the idea of starting with a triangulation of a sphere and performing edge subdivisions to make it into a flag triangulation, while of course taking care to ensure that we preserve a lower bound on the chromatic number. An edge subdivision is formally defined as follows.
\begin{definition}
For an edge $e = \{u, v\}$ in a simplicial complex $X$ we define the \emph{subdivision of $X$ at $e$}, denoted $X(e)$ to be the simplicial complex obtained from $X$ by adding a new vertex $w$ and replacing each facet of the form $\{u, v\} \sqcup \sigma$ with the facets $\{u, w\} \sqcup \sigma$ and $\{v, w\} \sqcup \sigma$.
\end{definition}
A flag complex is a simplicial complex where all minimal nonfaces have size 2. So if for example $X$ has an empty triangle we can subdivide an edge of that triangle, however that may create new empty triangles with the subdivision vertex as in Figure \ref{SubdivisionFigure}. It isn't clear then that this edge subdivision process can be done in such a way as to guarantee it will terminate in a flag complex in a finite number of steps. Moreover, we are also interested in bounding the number of required subdivisions by a polynomial.

\begin{figure}
\centering
\begin{tikzpicture}[line cap=round,line join=round,>=triangle 45,x=3cm,y=3cm]
\clip(0.19760407760459328,-1.2943672235324037) rectangle (4.204698410934663,1.1787006800592954);
\fill[line width=1pt,fill=black,fill opacity=0.42] (2,-0.42264973081037444) -- (1,-1) -- (2,0.7320508075688775) -- cycle;
\fill[line width=1pt,fill=black,fill opacity=0.43] (2,0.7320508075688775) -- (2,-0.42264973081037444) -- (3,-1) -- cycle;
\draw [line width=1pt] (2,0.7320508075688775)-- (2,-0.42264973081037444);
\draw [line width=1pt] (2,-0.42264973081037444)-- (1,-1);
\draw [line width=1pt] (1,-1)-- (2,0.7320508075688775);
\draw [line width=1pt] (2,0.7320508075688775)-- (3,-1);
\draw [line width=1pt] (2,-0.42264973081037444)-- (3,-1);
\draw [line width=1pt] (1,-1)-- (3,-1);
\draw [line width=1pt] (2,-0.42264973081037444)-- (2.792689239615456,0.5910184850391601);
\draw [line width=1pt] (2,0.7320508075688775)-- (2.792689239615456,0.5910184850391601);
\draw (1.9515162519091687,0.8760344198916922) node[anchor=north west] {$u$};
\draw (1.9515162519091687,-0.4303455919428352) node[anchor=north west] {$v$};
\draw (0.8650219846408477,-0.9942878544773441) node[anchor=north west] {$x$};
\draw (3.0364663214154695,-0.9942878544773441) node[anchor=north west] {$y$};
\draw (2.8491674679618053,0.6690831308882026) node[anchor=north west] {$z$};
\draw [line width=1pt] (2,0.15470053837925155)-- (1,-1);
\draw [line width=1pt] (2,0.15470053837925155)-- (3,-1);
\draw (1.831670191870491,0.2603543351063109) node[anchor=north west] {$w$};
\begin{scriptsize}
\draw [fill=black] (1,-1) circle (2.5pt);
\draw [fill=black] (3,-1) circle (2.5pt);
\draw [fill=black] (2,0.7320508075688775) circle (2.5pt);
\draw [fill=black] (2,-0.42264973081037444) circle (2.5pt);
\draw [fill=black] (2.792689239615456,0.5910184850391601) circle (2.5pt);
\draw [fill=black] (2,0.15470053837925155) circle (2.5pt);
\end{scriptsize}
\end{tikzpicture}
\caption{Subdividing $uv$ with $w$ eliminates $uvz$ as an empty triangle, but adds $wxy$ as an empty triangle.}\label{SubdivisionFigure}
\end{figure}
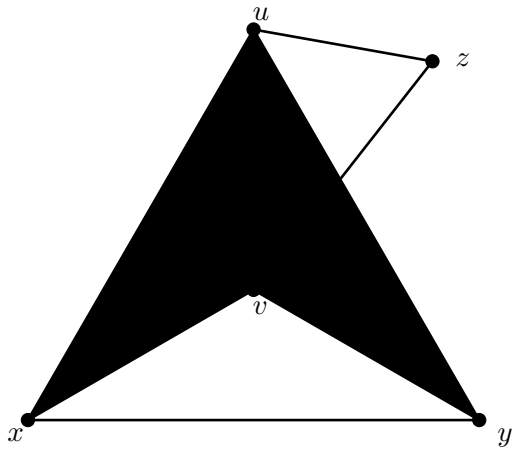

We begin with a lemma that makes a some observations about how the structure of empty simplices are changed by an edge subdivision. 
\begin{lemma}\label{NewEmpty}
Suppose that $X$ is a simplicial complex in which every minimal nonface has dimension at most $d$. Then a subdivision at an edge $e = \{u, v\}$ with subdivision vertex $w$ cannot create a minimal nonface of dimension larger than $d$ and any new minimal nonface in $\mnf(X(e)) \setminus \mnf(X)$ is of the form $\{w\} \sqcup \tau$ with either $\tau \sqcup \{v\}$ or $\tau \sqcup \{u\}$ as a minimal nonface in $\mnf(X)$.
\end{lemma}
\begin{proof}
Suppose that $\sigma$ is a minimal nonface created by subdividing at $e$ of dimension at least $d$. If $w \notin \sigma$ then $\partial \sigma$ belongs to $X(e)$, but does not contain $w$ so it belongs to $X$ as well, thus $\sigma$ was also a minimal nonface in $X$, so in particular it has dimension at most $d$. If $w \in \sigma$ then take $\tau = \sigma \setminus w$ with $|\tau| \geq d$. If $u$ or $v$ belongs to $\tau$, without loss of generality suppose that $u$ belongs to $\tau$, then $\{w\} \sqcup \tau \setminus \{u\}$ is a face of $X(e)$ by minimality. By how the edge subdivisions are done then $\{v\} \sqcup \tau$ is a face of $X$ (observing that no face of $X(e)$ contains both $u$ and $v$), but then $\{w\} \sqcup \tau$ is a face of $X(e)$ as it obtained by deleting $u$ and $v$ from $\{v\} \sqcup \tau$ and then adding $\{u, w\}$, so this is a contradiction. Thus we assume $u$ and $v$ do not belong to $\tau$. By minimality $\tau$ belongs to $X(e)$ and does not contain $w$ so $\tau$ belongs to $X$. Moreover the cone over $\partial \tau$ with cone point $w$ was added in the subdivision, so $\partial \tau$ is in the link in $X$ of $\{u, v\}$. But since $w \sqcup \tau$ is a nonface of $X(e)$, $\{u, v\} \sqcup \tau$ is a nonface of $X$, however its dimension is at least $d + 1$ so it cannot be a minimal nonface. By how we create the subdivision every proper face of $\{u\} \sqcup \tau$ and $\{v\} \sqcup \tau$ belongs to $X$, so one of $\{u \} \sqcup \tau$ or $\{v\} \sqcup \tau$ is a minimal nonface of $X$.
\end{proof}

We use edge subdivisions to give a construction of a flag 3-sphere on $n$ vertices with chromatic number at least $\widetilde{\Omega}(n^{1/4})$. The idea of the construction will be to start with a graph $G$ with large chromatic number and then embed that graph inside the 1-skeleton of the boundary of the cyclic 4-polytope. Recall that the cyclic 4-polytope on $n$ vertices is the convex hull of $n$ points on the moment curve $(t, t^2, t^3, t^4)$ and by Gale's evenness its facets are all collections of $\{x, x + 1, y, y + 1\}$  for $x$ and $y$ on the $n$-cycle, in particular it is 2-neighborly so the 1-skeleton is the complete graph. At this point we have $G$ as a subgraph of a non-flag triangulation of a 3-sphere. Next we show that we can do edge subdivisions within this 3-sphere without changing the subgraph $G$ in a way that results in a flag triangulation of $S^3$ that still contains $G$ as a subgraph. For the lower bound of Theorem \ref{maintheorem} we use the following subdivision lemma and a construction from Ramsey theory.
\begin{lemma}\label{SubdivisionLemma}
If $G$ is a triangle-free graph on $n$ vertices then there exists a flag triangulation $X$ of $S^3$ on at most $4\binom{n}{2} + n$ vertices containing $G$ as a subgraph.
\end{lemma}
The only minimal empty simplices in a cyclic 4-polytope are empty triangles, see for example the discussion after Corollary 4.16 in \cite{Nagel}. By Lemma \ref{NewEmpty} it suffices to show that there is a sequence of edge subdivisions on non-edges of $G$ that eliminates all empty triangles. As we saw in Figure \ref{SubdivisionFigure} we do have to be careful as a single subdivision will not necessarily decrease the number of empty triangles. In our proof of Lemma \ref{SubdivisionLemma} we solve this problem by keeping track of empty triangles on original vertices, i.e. empty triangles that are on three vertices of the original cyclic 4-polytope. Our proof essentially shows that if at any point all the empty triangles are on original vertices then there is a sequence of at most four subdivisions that include a subdivision of an original edge and end with all empty triangles still on the original vertices. 
\begin{proof}[Proof of Lemma \ref{SubdivisionLemma}]
We regard $G$ as being on $n$ labelled vertex and embed $G$ in the cyclic 4-polytope on $[n]$. We now want to subdivide the edges of the cyclic polytope without ever subdividing an edge of $G$ so that we arrive a flag triangulation of $S^3$. We claim that this can be done in at most $4 \binom{n}{2}$ steps. As long as there is an empty triangle we will pick an edge of it to subdivide. Since $G$ is triangle free, we will never be required to subdivide an edge of $G$. However when we subdivide an edge we may create new empty triangles. Our procedure will get rid of those new empty triangles as soon as possible so that in a bounded number of steps we return to a situation where all empty triangles are on original vertices.

Suppose we have a partial subdivision of our cyclic polytope with $G$ embedded in it, and furthermore suppose that all empty triangles are on the original vertices of the cyclic polytope. Suppose that $\partial \tau$ is an empty triangle on the original vertices.  Then the vertices of $\tau$ do not contain a pair of vertices $\{i, i + 1\}$ in $[n]$ nor does $\tau$ contain $\{1, n\}$ since there are no empty triangles in the cyclic polytope that contain such a pair by Gale's evenness. This means that the link of each edge of $\tau$ is a cycle with at most 4 of the original vertices. Select one edge $e$ of $\tau$ that is not in $G$ to subdivide. Since all empty triangles that we have before subdividing $\tau$ involve original vertices, the only way that our subdivision vertex $w$ can belong to a new empty triangle is if it belongs to an empty triangle with two original vertices that are adjacent in the subdivided complex but not in the link of $e$. It is clear that there are at most two such pairs. Therefore when we subdivide $e$ with subdivision vertex $w$ we have at most two empty triangles that contain $w$ and those are the only empty triangles that contain a subdivided vertex. If the link of $e$ is just a 4-cycle on original vertices then we can take three more subdivision and no longer have any empty triangles containing a subdivided vertex, see Figure \ref{4Cycle}.
\begin{figure}[H]
\centering
\definecolor{ccqqqq}{rgb}{0.8,0,0}
\begin{tikzpicture}[line cap=round,line join=round,>=triangle 45,x=1cm,y=1cm]
\clip(-2.1290112999066135,-2.183676445443878) rectangle (2.167395061411194,1.0723445189628877);
\draw [line width=1pt] (-1.5,-0.5)-- (0,1);
\draw [line width=1pt] (0,1)-- (1.5,-0.5);
\draw [line width=1pt] (-1.5,-0.5)-- (0,-2);
\draw [line width=1pt] (0,-2)-- (1.5,-0.5);
\draw [line width=1pt] (0,-0.5)-- (-1.5,-0.5);
\draw [line width=1pt] (0,-0.5)-- (1.5,-0.5);
\draw [line width=1pt] (-0.5,0)-- (-1.5,-0.5);
\draw [line width=1pt] (-0.5,0)-- (0,-0.5);
\draw [line width=1pt] (-0.5,0)-- (0,1);
\draw [line width=1pt] (0,-0.5)-- (0,-2);
\draw [line width=1pt] (0,-0.5)-- (0,1);
\draw [line width=1pt] (-0.75,-0.5)-- (0,-1.25);
\draw [line width=1pt] (-0.75,-0.5)-- (-0.5,0);
\draw [line width=1pt] (-0.75,-0.5)-- (0,-2);
\draw [line width=1pt] (0,-1.25)-- (1.5,-0.5);
\draw (0.04107406476843661,-0.13180524598044277) node[anchor=north west] {$w$};
\begin{scriptsize}
\draw [fill=black] (0,-2) circle (2.5pt);
\draw [fill=black] (1.5,-0.5) circle (2.5pt);
\draw [fill=black] (0,1) circle (2.5pt);
\draw [fill=black] (-1.5,-0.5) circle (2.5pt);
\draw [fill=ccqqqq] (0,-0.5) circle (2pt);
\draw [fill=ccqqqq] (-0.5,0) circle (2pt);
\draw [fill=ccqqqq] (0,-1.25) circle (2pt);
\draw [fill=ccqqqq] (-0.75,-0.5) circle (2pt);
\end{scriptsize}
\end{tikzpicture}

\caption{Subdivision in the simplest case.} \label{4Cycle}

\end{figure}
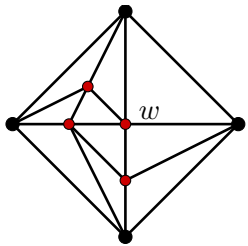

Otherwise the link of $e$ is a cycle with at most 4 original vertices. Now suppose that subdividing $e$ creates the empty triangle on vertices $w, o_1, o_3$. Then the link of the edge $\{w, o_1\}$ must be a four cycle on the two vertices in link of $w$ adjacent to $o_1$ and the two end points of $e$. Let $x$ and $y$ denote the neighbors of $o_1$ in the link of $e$. As the endpoints of $e$ are no longer adjacent after subdivision the only way that we could create an empty triangle subdividing $\{w, o_1\}$ with subdivision vertex $w'$ is if we create the empty triangle $\{w', x, y\}$ but then by how we do the subdivision we'd have either $\{o_1, x, y\}$ or $\{w, x, y\}$ as an empty triangle before we subdivide $\{w, o_1\}$. This means then that $x$ and $y$ must be original vertices. If this is the case then we instead try to subdivide $\{w, o_3\}$ unless the $x$ and $y$ in that case are also original vertices. However if this happens we are in the case pictured in Figure \ref{4Cycle} and we know what to do. If necessary after subdividing to get rid of the empty triangle on $\{w, o_1, o_3\}$ we do the same thing with one more subdivision to get rid of the other empty triangle $\{w, o_2, o_4\}$. At this point all the empty triangles contain only original vertices. It is clear that after at most $\binom{n}{2}$ repetitions of this procedure we no longer have any empty triangles because subdividing all edges of the complete graph outside of $G$ in this way will get rid of all empty triangles on original vertices without creating new empty triangles. Within each subdivision procedure we perform at most four edge subdivisions giving us that at most $4\binom{n}{2}$ subdivisions are sufficient until we have subdivided at least one edge of every empty triangle we started with while also having all empty triangle on the original vertices. Therefore we have no empty triangles.
\end{proof}

Next we use the following result of Kim to find a triangle-free graph $G$ on which we can apply Lemma \ref{SubdivisionLemma}. 

\begin{theorem}[Kim \cite{JHKim}] \label{KimsConstruction}
For $n$ large enough there exists a triangle-free graph on $n$ vertices with chromatic number $\Omega \left(\frac{\sqrt{n}}{\sqrt{\log n}} \right)$
\end{theorem}

The following more precise statement for the lower bound of Theorem \ref{maintheorem} follows immediately from Theorem \ref{KimsConstruction} and Lemma \ref{SubdivisionLemma}.
\begin{theorem}
For any $\epsilon > 0$, there exists an absolute constant $C$ so that for each $m$ there is a flag triangulation of $S^3$ on at most $C m^4 \log^{2 + \epsilon} m$ vertices that is not $m$ colorable.
\end{theorem}
\begin{proof}
By Theorem \ref{KimsConstruction}, there is $C_1$ so that for $m$ large enough there is a graph $G$ on at most $C_1 m^2 \log^{1 + \epsilon} m$ vertices with chromatic number at least $m$. Now by Lemma \ref{SubdivisionLemma}, $G$ is a subgraph of a flag triangulation of $S^3$ with at most 
\[2 (C_1 m^2 \log^{1 + \epsilon} m)^2\]
vertices.
\end{proof}

\section{Concluding remarks}
A necessary condition for a counterexample to Chudnovsky and Nevo's conjecture for $d = 3$ would be a flag 3-sphere with high chromatic number. How then does the construction in the proof of Theorem \ref{maintheorem} compare with their conjecture? Within the original vertices the graph $G$ from Theorem \ref{KimsConstruction} has no linear sized independent set, see the original proof in \cite{JHKim}. However, since the cyclic polytope has $\Theta(n^3)$ empty triangles and each edge belongs to at most $O(n)$ empty triangles, we require at least $\Omega(n^2)$ subdivisions to get rid of all the empty triangles. By that point the original vertices are such a small portion of the vertices that we wouldn't expect them to have a meaningful effect on the size of the largest independent set; it seems unlikely then that this construction would provide a counterexample to the conjecture. However, the simplest examples of flag 3-spheres are barycentric subdivisions of arbitrary 3-spheres and joins of flag triangulations of smaller dimensional spheres. Any barycentric subdivision of a $d$-complex clearly has chromatic number $d + 1$, while the suspension of a flag 2-sphere will have chromatic number at most 5 and the join of cycles will have chromatic number at most 6. Since the Chudnovsky--Nevo Conjecture claims the existence of a maximal independent of size at least $n/6$ asymptotically for flag 3-spheres, it is necessary to consider examples beyond these simple ones. 

There are a few other directions of future research as well. One question is related to the study of the high-dimensional chromatic numbers of simplicial complexes as in \cite{LutzMoller, HPPT, LeeNevo} namely:
\begin{question}
Is there a meaningful class of complexes $\mathcal{X}$ so that for some $k$, $\chi_k(X)$ is unbounded for all $X \in \mathcal{X}$, but $\chi_k(X)$ is bounded for flag complexes in $\mathcal{X}$?
\end{question}
Additionally there are the quantitative questions around the chromatic number when it is unbounded. Our main theorem provides bounds for the following question.
\begin{question}
What is the right rate of growth of $\chi(n, d)$ for $d$ fixed and $n$ tending to infinity?
\end{question}
Lastly, it would be interesting to look for a topological interpolation between bounds on chromatic numbers (or independent sets) in flag manifolds and classical Ramsey theory. This idea was suggested by Florian Frick. Let $R_b(k, t)$ denote the maximum number of vertices in a $K_k$ free graph without an independent set of size $t$ where the $(k-1)$st Betti number is at most $b$. Oriented flag manifolds are a special case for when $b = 1$, and $b = \infty$ is ordinary Ramsey theory. How does $R_b(k, t)$ grow for $k$ and $b$ fixed?


\bibliography{ResearchBibliography}
\bibliographystyle{amsplain}
\end{document}